\documentclass[final]{amsart}
\usepackage{amsmath,amsfonts,amsthm,amssymb}
\usepackage{xcolor,comment,soul}
\usepackage[color]{showkeys}
\usepackage{tikz}

\definecolor{refkey}{rgb}{0,1,1}
\definecolor{labelkey}{rgb}{1,0,0}

\newcommand{\C}{\mathbb{C}}
\newcommand{\R}{\mathbb{R}}
\newcommand{\conv}{\text{conv}\,}
\newcommand{\tr}{\text{tr}\,}
\newcommand{\inter}{\operatorname{int}}
\renewcommand{\span}{\text{span}\,}
\newcommand{\eq} [1] {\begin{equation}\label{#1}}
\newcommand{\en} {\end{equation}}

\renewcommand{\P} {\mathbb{P}}
\newcommand{\ds}{\displaystyle}

\newcommand{\CS}{\mathbb{C}S}
\renewcommand{\Im}{\operatorname{Im}}
\newcommand{\imag}{\Im}
\renewcommand{\Re}{\operatorname{Re}}
\newcommand{\real}{\Re}
\newcommand{\diag}{\operatorname{diag}}
\newcommand{\abs}[1]{\left\vert#1\right\vert}

\newcommand{\yp}{y^+(\alpha_k)}
\newcommand{\ym}{y^-(\alpha_k)}

\newtheorem{thm}{Theorem}
\newtheorem{theorem}[thm]{Theorem}

\newtheorem{lemma}[thm]{Lemma}
\newtheorem{prop}[thm]{Proposition}
\newtheorem{cor}[thm]{Corollary}

\theoremstyle{definition}
\newtheorem{example}[thm]{Example}

\begin{document}

\title[Continuity of Vectors Realizing Points in the Field of Values]{Continuity Properties of Vectors Realizing Points in the Classical Field of Values}
\author{Dan Corey, Charles R. Johnson, Ryan Kirk, Brian Lins$^*$, Ilya Spitkovsky}
\date{}
\address{Dan Corey, University of Notre Dame}
\email{dcorey@nd.edu}
\address{Charles R. Johnson, College of William \& Mary}
\email{crjohnso@math.wm.edu}
\address{Ryan Kirk, University of North Carolina at Chapel Hill}
\email{rtkirk@email.unc.edu}
\address{Brian Lins, Hampden-Sydney College}
\email{blins@hsc.edu}
\thanks{$^*$Corresponding author.}
\address{Ilya Spitkovsky, College of William \& Mary}
\email{ilya@math.wm.edu, imspitkovsky@gmail.com}
\thanks{This work was partially supported by NSF grant DMS-0751964}

\begin{abstract}
For an $n$-by-$n$ matrix $A$, let $f_A$
be its ``field of values generating function'' defined as
$f_A\colon x\mapsto x^*Ax$. We consider two natural versions of the
continuity, which we call strong and weak, of $f_A^{-1}$ (which is of course multi-valued) on the field of values $F(A)$.  The strong continuity holds, in particular, on the interior of $F(A)$, and at such points $z \in \partial F(A)$ which are either corner
points, belong to the relative interior of flat portions of $\partial
F(A)$, or whose preimage under $f_A$ is contained in a one-dimensional set.
Consequently, $f_A^{-1}$ is continuous in this sense on the
whole $F(A)$ for all normal, 2-by-2,  and unitarily irreducible 3-by-3
matrices. Nevertheless, we show by example that the
strong continuity of $f_A^{-1}$ fails at certain points of $\partial F(A)$ for
some (unitarily reducible) 3-by-3  and (unitarily irreducible) 4-by-4
matrices. The weak continuity, in its turn, fails for some unitarily reducible 4-by-4 and untiarily irreducible 6-by-6 matrices.
\end{abstract}

\maketitle
\section{Introduction}
For each $A\in M_{n}(\C)$, the \textit{field of values} is
defined by
\[ F(A)=\{x^*Ax \colon x\in\C^n,x^*x=1\}.\]
The field of values arises in many contexts and it and its
generalizations have been heavily studied. Many known properties
may be found, e.g.,  in \cite{HJ2} or \cite{GusRa}.

Let $f_{A}(x)=x^{*}Ax$. Of course, $f_{A}$ is a continuous function of
$x$ on $\C^n$, and $F(A)$ is the image under $f_A$ of the  unit sphere
$\C{S}^n$ of $\C^n$. The inverse map from $F(A)$ to the unit
sphere is  multi-valued, in particular because $f_A(\omega x)=f_A(x)$
for any unimodular $\omega\in\C$. It is therefore (more) natural to
consider $f_A$ as defined on the quotient space $\C S^n/\C S^1$
which can be identified with the complex projective space
$\C\P^{n-1}$. But even under this convention, the mapping \eq{inv}
f_A^{-1}\colon F(A)\longrightarrow \C\P^{n-1} \en for $n>1$ is never
one-to-one. In fact, for $z \in \inter F(A)$, $f^{-1}_A(z)$ contains $n$ linearly independent vectors \cite[Theorem 1]{Carden09}.

Here, we consider the continuity properties of the multi-valued
mapping \eqref{inv}.  There are various notions of continuity of such
mappings, see e.g. \cite[Section c-9]{Hart04}. Let  $g$ be a multi-valued function from a metric
space $(X,d)$ to a metric space $(Y,\rho)$. Then, borrowing the
terminology from \cite{BraHe94},  $g$ is {\em weakly continuous} at
$x\in X$ if there exists $y\in g(x)$ for which 
\eq{cont} 
\forall \ \epsilon>0\ \exists\ \delta>0\colon d(x,x')<\delta\Rightarrow \exists\ y'\in g(x') \text{ such that } \rho(y,y')<\epsilon. 
\en  
Respectively,  $g$ is {\em strongly continuous} at
$x\in X$ if \eqref{cont} holds {\em for all} $y\in g(x)$.

Note that $g$ is weakly (strongly) continuous at $x$ if and only if the inverse mapping
$f=g^{-1}$ is open at some (respectively, all) $y\in g(x)$. Of course, if
$g(x)$ happens to be a singleton, then weak continuity of $g$ at $x$ is equivalent to its strong continuity there.

\begin{prop}\label{comp}
Let $g=f^{-1}$, for $f\colon Y\rightarrow X$
a continuous single valued function defined on a compact set $Y$. Then
$g$ is strongly continuous at all points $x\in X$ for which $g(x)$ is a singleton.
\end{prop}
For injective $f$ Proposition~\ref{comp} means simply that its inverse $g$ is strongly continuous on $X$. This ``global'' result is standard, and the proof of the ``local'' version is literally the same.

Due to the already mentioned multivalued nature of $f_A^{-1}$, the applicability of Proposition \ref{comp} in our setting is rather limited: it can possibly be applicable only to boundary points of $F(A)$ and, as we will see, not all of them. So, in Section~\ref{section:main} we propose an alternative approach. With the help of the latter, we show that $f_A^{-1}$ is strongly continuous on the interior of $F(A)$ for all $A$, and on the whole $F(A)$ for convexoid (in particular, normal) and all 2-by-2 matrices. Examples of matrices of bigger size, for which strong continuity fails while weak continuity persists, or even weak continuity fails, are given in Section \ref{section:further}.

\section{Main results} \label{section:main}
It is convenient for us to express $f_A(x) = x^*A x$, and allow the
domain of $f_A$ to be either $\CS^n$ or $\C \P^{n-1}$ under the
convention that $x$ can be any representative of the equivalence
class of unit vectors under unimodular scaling. The following
theorem will imply the strong continuity of $f_A^{-1}$ at many $z \in
F(A)$.

\begin{theorem} \label{scaling}
Let $A \in M_n(\C)$ and suppose that $z = x^*Ax$, in which $x \in \C
S^n$. For any neighborhood $U$ of $x$ in $\C S^n$ there is a
constant $\delta > 0$ such that $ \delta F(A) + (1-\delta)z \subseteq
f_A(U)$.
\end{theorem}

To prove Theorem \ref{scaling}, we use the following lemma
concerning the image of a spherical cap under a linear
transformation.
\begin{lemma} \label{sphericalCap}
Let $S$ be the surface of a sphere with radius $r$ in $\R^3$ and let $x
\in S$.  For any $\epsilon \in [0, 2]$, let $C_\epsilon = \{y \in S: ||x-y|| \le
\epsilon r\}$.  If $T\colon \R^3 \rightarrow \R^2$ is a linear
transformation, then $\frac{1}{4}\epsilon^2 T(S) + (1 -
\frac{1}{4}\epsilon^2)T(x) \subset T(C_\epsilon)$.
\end{lemma}
\begin{proof}
It is sufficient to prove that $T(C_\epsilon) = T(\conv C_\epsilon)$
since $\frac{1}{4}\epsilon^2 S + (1 - \frac{1}{4}\epsilon^2)x \subset
\conv C_\epsilon$.
Let $\partial C_\epsilon = \{y \in S : ||x-y|| = \epsilon r \}$.  Since
$\partial C_\epsilon$ is a circle in $\R^3$, $T(\partial C_\epsilon)$ is
either a line segment or an ellipse.  Any points inside the ellipse
$T(\partial C_\epsilon)$ are contained in $T(C_\epsilon)$ because
$T(C_\epsilon)$ must be simply connected.  Thus $T(C_\epsilon) =
T(C_\epsilon \cup \conv \partial C_\epsilon)$.  The set $C_\epsilon
\cup \conv \partial C_\epsilon$ is homeomorphic to $S^2$, so it
separates $\R^3$ into interior and exterior components.  The union of
$C_\epsilon \cup \conv \partial C_\epsilon$ with its interior
component is precisely $\conv C_\epsilon$.  It follows that $T(\conv
C_\epsilon) = T(C_\epsilon \cup \conv \partial C_\epsilon) =
T(C_\epsilon)$.  
\end{proof}

\begin{proof}[Proof of Theorem \ref{scaling}]
Let $\Sigma = \{yy^* : y \in \C S^n \}$.  The set $W = \{yy^* : y \in U\}$
is a neighborhood around $xx^*$ in $\Sigma$.  Therefore, there exists
$\epsilon > 0$ such that the set $B_\epsilon = \{yy^* \in \Sigma :
||xx^*-yy^*|| \le \frac{\epsilon}{2} \}$ is contained in $W$.

Choose any $y \in \C S^n$. Let $V = \span \{x,y\}$.  For linearly
independent $x$ and $y$, the set $\Omega_V = \span \{v v^* : v \in V
\}$ is a subspace with 4 real dimensions.  The intersection of
$\Omega_V$ with $\Sigma$ is the surface of a sphere with radius
$\frac{1}{2}$ in the 3 dimensional affine subspace consisting of
matrices with trace 1 \cite{Dav71}.  Let $\hat{f}_A$ denote the linear
map $X \mapsto \tr(AX)$ where $X$ is any $n$-by-$n$ Hermitian
matrix. Note that $f_A(v) = \hat{f}_A(vv^*)$ for all $v \in \C^n$.  By
Lemma \ref{sphericalCap}, $\frac{\epsilon^2}{4} \hat{f}_A(yy^*) +
(1-\frac{\epsilon^2}{4})\hat{f}_A(xx^*) \in  \hat{f}_A(B_\epsilon \cap
\Omega_V) \subset \hat{f}_A(W)$. Therefore $\frac{\epsilon^2}{4}
f_A(y) + (1-\frac{\epsilon^2}{4}) f_A(x) \in f_A(U)$.  Since $\epsilon$
does not depend on the choice of $y$, we have shown that
$\frac{\epsilon^2}{4} F(A) + (1-\frac{\epsilon^2}{4}) z \subseteq
f_A(U)$.
\end{proof}
According to Theorem~\ref{scaling}, the openness of the mapping
$f_A$ at $x$ is guaranteed whenever the scaled sets $z+\delta(F(A)-z)$
contain a full neighborhood of $z=f_A(x)$  for all $\delta>0$.
Consequently:
\begin{thm}\label{th:cont}
Let $A \in M_n(\C)$. Then $f_A^{-1}$ is strongly continuous on $F(A)$ except, perhaps, at the round points of its boundary $\partial F(A)$.
\end{thm}

Note that, in our terminology, $z\in\partial F(A)$ is {\sl not} a round point of $\partial F(A)$ if and
only if {\em both} one-sided neighborhoods of $z$ in $\partial F(A)$
are line segments when the radius of the neighborhood is sufficiently
small. For {\em convexoid} matrices \cite{Hal82}, by definition, $F(A)$
is the convex hull of the spectrum of $A$ and thus a polygon. There
are no round point  in $\partial F(A)$ in this case, so that the following
statement holds.
\begin{cor}\label{co:norm}Let $A \in M_n(\C)$ be convexoid. Then $f_A^{-1}$ is strongly continuous on $F(A)$.\end{cor}
Of course, this result covers all normal matrices $A$.

It is possible for the mapping $f_A^{-1}$ to be strongly continuous even at the
round points of $F(A)$. According to Proposition~\ref{comp}, this will be the case, in particular, when the preimage of $z \in F(A)$ { is} contained in a one-dimensional set. As was
observed, e.g., in \cite{Tsi84}  and \cite{RS11}, this property is
possessed by all boundary points of $F(A)$  for non-normal 2-by-2 matrices $A$.
Combining this observation with the result for normal matrices, we
arrive at
\begin{cor}\label{co:n2}
The mapping $f_A^{-1}$ is strongly continuous for all $2$-by-$2$ matrices
$A$.
\end{cor}

The latter result can, of course, be proved directly:
\begin{proof}[Direct proof of Corollary \ref{co:n2}]
The map $x \mapsto xx^*$ is a continuous bijection from the compact
space $\C\P^1$ onto $S = \{xx^* : x \in \C^2, x^*x=1\}$. Thus $\C\P^1$
and $S$ are homeomorphic.  As noted in \cite{Dav71}, $S$ is a
2-sphere and the real linear transformation $\hat{f}_A(X) = \tr(AX)$
maps $S$ onto $F(A)$. In the case where the field of values is an
ellipse (i.e., $A$ is not normal), the linear transformation $\hat{f}_A$ is
a composition of an orthogonal projection onto a two dimensional
subspace composed with an invertible real linear transformation of
that subspace into $\C$.  The projection divides $S$ into a union of
two closed hemispheres that are each mapped bijectively onto
$F(A)$.  Since $\hat{f}_A$ is continuous and $S$ is compact, it follows
that $\hat{f}_A$ is a homeomorphism on each hemisphere.

If $F(A)$ is a line segment, then $\hat{f}_A$ maps two antipodal points
of $S$ onto the two endpoints of $F(A)$.  Any geodesic arc of $S$
connecting the two antipodal points is mapped bijectively onto
$F(A)$.  Since $\hat{f}_A$ is continuous and $S$ is compact, it follows
that $\hat{f}_A$ is a homeomorphism on each such arc.
\end{proof}

\section{Further observations} \label{section:further}

Starting with $n=3$, the strong continuity of $f_A^{-1}$ may
indeed fail at round points of $\partial F(A)$ having multiple linearly
independent preimages under $f_A$.
\begin{thm}\label{th:nocont}
Let $A_1$, $A_2$ be two matrices such that $F(A_1)$
and $F(A_2)$ have a common support line at some point $z\in
\partial F(A_1)\cap \partial F(A_2)$ but not at any other point in some neighborhood of $z$.
Then {\em : (i)} $f_A^{-1}$ is not strongly continuous at $z$ for $A=A_1\oplus A_2$. {\em (ii)}
If in addition $z$ is a limit point for both $\partial F(A)\cap\partial F(A_1)$ and $\partial F(A)\cap\partial F(A_2)$, then $f_A^{-1}$ is not even weakly continuous at $z$.
\end{thm} 
\begin{proof} (i) Take $x_j\in f_{A_j}^{-1}(z)$, $j=1,2$, and consider a neighborhood $U$ of $z$ so small
that $\partial F(A_1)\cap\partial F(A_2)\cap U=\{z\}$. Switching the
indices $j=1,2$ if needed, we without loss of generality may suppose
that a one-sided neighborhood $\gamma$ of $z$ in say $\partial
F(A_1)$ lies outside of $F(A_2)$. But then all the vectors in
$f_A^{-1}(\gamma)$ must have zero second component, and are
therefore separated from $0\oplus x_2$. This violates \eqref{cont} at
$z$.

(ii) If the additional condition holds, then there also exists a one sided neighborhood $\widetilde{\gamma}$ of $z$ in $\partial
F(A_2)$ that lies outside of $F(A_1)$. Consequently,  all the vectors in
$f_A^{-1}(\widetilde{\gamma})$ must have zero first component. So, for any choice of $x\in f_A^{-1}(z)$ and its small neighborhood in $\C S^n$ the set $f_A(U)$ will miss either $\gamma$ or $\widetilde{\gamma}$ and thus will not be relatively open in $F(A)$.\end{proof}

{ Of course, we need $n\geq 3$ in order for conditions of Theorem~\ref{th:nocont} to hold ($n\geq 4$ for its part (ii)). \begin{example}\label{ex:3x3} Let $A_1$ be a 2-by-2 not normal matrix and $A_2=[z]$, where $z\in\partial F(A_1)$. Then all vectors $x\in f_A^{-1}(\zeta)$ with $\zeta\in\partial F(A_1)$, $\zeta\neq z$ have the zero third coordinate. Therefore, they lie at the distance $\sqrt{2}$ from
$[0,0,1]^T$, one of the vectors in $f_A^{-1}(z)$. Thus, $f_A^{-1}$ is not strongly continuous at the point $z$. Note however that, since $f_{A_1}^{-1}$ is strongly continuous on $F(A_1)=F(A)$, the mapping $f_A$ is open at $[f_{A_1}^{-1}(z), 0]^T$, so that $f_A^{-1}$ is weakly continuous at $z$. \end{example} }

The simplest example of this sort is delivered by \[
A_1=\left[\begin{matrix} 0 & 2\\ 0 & 0\end{matrix}\right],  A_2=[1], \text{
so that } A=\left[\begin{matrix} 0 & 2 & 0\\ 0 & 0 & 0 \\ 0 & 0 &
1\end{matrix}\right] { \text{ and } z=1}. \]

{ \begin{example}\label{ex:4x4}  Let \[ A_1 = \left[\begin{matrix} 0 & ik \\ ik & 1+ib\end{matrix}\right], \quad A_2=\left[\begin{matrix} 0 & ik \\ ik & 1-ib\end{matrix}\right] \text{ with } b,k>0, \text { and } A=A_1\oplus A_2. \] Then some deleted neighborhood of $z=0$ in the intersection of $\partial F(A)$ with the upper (lower) half plane lies in $F(A_1)\setminus F(A_2)$ (respectively, $F(A_2)\setminus F(A_1)$). Thus, weak continuity of $f_A^{-1}$ fails at $z=0$.  \end{example} }

\begin{figure}[ht]
\includegraphics[scale=1]{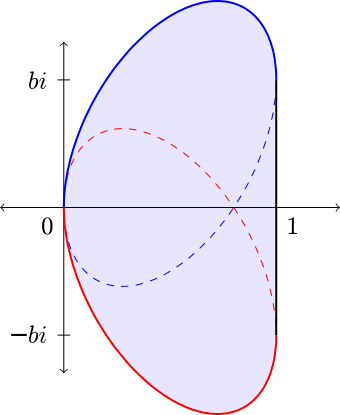}

\caption{$F(A)$ from Example \ref{ex:4x4} with dashed lines indicating the boundaries of $F(A_1)$ and $F(A_2)$.  Weak continuity fails at $z=0$. }
\end{figure}

Recall that a matrix is \textit{unitarily reducible}, see e.g. \cite[p. 60]{HJ2}, if and only if there exist a unitary matrix $U$ such that $U^*AU$ is the direct sum of two smaller matrices.  All matrices satisfying conditions of Theorem~\ref{th:nocont} are
unitarily reducible by construction. { As it happens, for $n=3$ strong continuity of $f_A^{-1}$ actually can fail only for unitarily reducible matrices.}
\begin{thm}\label{th:n3}Let $A$ be an unitarily irreducible $3$-by-$3$ matrix. Then $f_A^{-1}$
is strongly continuous on $F(A)$.
\end{thm} \begin{proof} Suppose not. Then there is a point $z\in F(A)$ at which the strong continuity fails,
and it can only be a boundary round point with $f_A^{-1}(z)$
containing two linearly independent vectors $\xi,\eta$. By translation,
we may without loss of generality arrange that $z=0$, and by further
scaling --- that $F(A)$ is located in the right half plane. In other words,
$A=H+iK$, where $H,K$ are self adjoint, and $H$ in addition is positive
semi-definite.

Let ${\mathcal L}$ be the span of $\xi,\eta$, and let $B$ be the
compression of $A$ onto $\mathcal L$. Since $0\in F(B)\subset F(A)$,
zero must be a boundary round point of $F(B)$. But $\mathcal L$ is
two-dimensional, so that the only way in which a round point in $\partial F(B)$ can have more than one linearly independent preimage under the field of values generating
function is if $B$ is a scalar multiple of the identity. Consequently,
under an appropriate unitary similarity we
get \eq{HK} H, K = \left[\begin{matrix} 0 & 0 & \star \\ 0 & 0 & \star \\
\star & \star & \star\end{matrix}\right].\en Since $H$ is semi-definite,
\eqref{HK} further implies
that \eq{H} H = \left[\begin{matrix} 0 & 0 & 0 \\ 0 & 0 & 0 \\
0 & 0 & \star\end{matrix}\right].\en Yet another unitary similarity,
involving the first two rows and columns only, allows to put $K$ in the
form \eq{K}  K = \left[\begin{matrix} 0 & 0 & 0 \\ 0 & 0 & \star \\
0 & \star & \star\end{matrix}\right]\en without changing \eqref{H}.
Comparing \eqref{H} and \eqref{K} we observe that $H$ and $K$ have
a common eigenvector, which makes $A$ unitarily reducible.
\end{proof}

Observe that the weak continuity persists for all 3-by-3 matrices, unitarily irreducible or not.
\begin{thm} \label{th:3x3}Let $A$ be a 3-by-3 matrix. Then the mapping $f_A^{-1}$ is strongly continuous on $F(A)$ except perhaps at one point $z\in\partial F(A)$ where it is weakly continuous.\end{thm}
\begin{proof} Due to Theorem~\ref{th:n3}, we need to consider only unitarily reducible matrices. By Corollary~\ref{co:norm}, we have the normal case covered. Thus, without loss of generality $A=A_1\oplus [z]$ for some non-normal 2-by-2 matrix $A_1$ and a number $z$. Since  $f_A^{-1}(\zeta)$ contains only linearly dependent vectors for all $\zeta\in \partial F(A)\setminus\{z\}$, the strong continuity of $f_A^{-1}$ is guaranteed  on $\inter F(A)\cup (\partial F(A)\setminus\{z\})$.  Finally, if $z\in\partial F(A)$, then strong continuity of $f_A^{-1}$ fails at $z$ while its weak continuity persists, as was shown in Example~\ref{ex:3x3}.  
\end{proof}

The proof of Theorem~\ref{th:n3} is based on the observation that for a
unitarily irreducible 3-by-3 matrix $A$ the preimages $f_A^{-1}(z)$ are
one dimensional for all round points $z\in \partial F(A)$. In higher dimensions,
however, there even exist unitarily irreducible matrices
$A$ for which all points of $\partial F(A)$ are generated by several
linearly independent vectors; see \cite{LSS} for respective examples.

Starting with $n=4$ it is possible for the strong continuity of
$f_A^{-1}$ to fail at a point of $\partial F(A)$ while
$A$ is unitarily irreducible.

\begin{example} \label{ex:UI4}
For $k_1, k_2, r >0$, { $k_1>k_2$,} let \[ A=\begin{bmatrix} 0 & 0 & ik_1 & 0 \\ 0 & 0 & 0 & ik_2 \\ ik_1 & 0 & 1 & ir  \\
0 & ik_2 & ir & 1 \end{bmatrix}= H+iK, \] where \[
H=\diag[0,0,1,1] \text{ and } K=\left[\begin{matrix}0 & 0 & k_1 & 0 \\  0 & 0
& 0 & k_2\\ k_1 & 0 & 0 & r\\ 0 & k_2 & r & 0\end{matrix}\right].\]  The matrix
$A$ is unitarily irreducible. Indeed, the eigenvectors of
$H$ are of the form \eq{ev} v_1=[x_1,x_2,0,0]^T \text{ and }
v_2=[0,0,x_3,x_4]^T \en  and thus are different from the eigenvectors of
$K$. Consequently, $H$ and $K$ do not have common
one-dimensional (thus, also three-dimensional) subspaces. A
two-dimensional common invariant subspace $\mathcal L$, if it
existed, would have to be spanned by non-zero $v_1,v_2$ of the form
\eqref{ev}. But then  $HKHv_2$ and  $HK^2Hv_2$, lying in $\mathcal
L$ and having the first two coordinates equal zero,  must be scalar
multiples of $v_2$. A direct computation shows that this is possible
only if $v_2=0$ which is a contradiction.

Observe further that  for any vector $x = [x_1,x_2,x_3,x_4]^T \in
\CS^4$,
$$f_A(x) = x^*Ax = 2i k_1 \Re(x_1 \bar{x}_3) + 2i k_2 \Re(x_2 \bar{x}_4) +2i
r\Re(x_3 \bar{x}_4) + |x_3|^2 + |x_4|^2.$$ Since $F(A)$ is the image of
$\CS^4$ under $f_A$, it follows that for any $z \in F(A)$, $0 \leq \Re z
\leq 1$. Note that $0 \in \partial F(A)$ { and $f_A^{-1}(0)$ is the unit sphere in the span of $e_1,e_2$, the first two vectors from the standard basis of $\C^4$}.

The boundary of $F(A)$ consists of the points $z=\alpha+i\beta$, where $\beta$ is one of the extremal values of $2k_1 \Re(x_1 \bar{x}_3) + 2k_2 \Re(x_2 \bar{x}_4) +2r\Re(x_3 \bar{x}_4)$ under the constraints $\abs{x_3}^2+\abs{x_4}^2=\alpha$, $\abs{x_1}^2+\abs{x_2}^2=1-\alpha$. For $\alpha<1$, the vector $x\in f_A^{-1}(z)$ must have a non-zero coordinate $x_1$, because otherwise the flip $x_1 \leftrightarrow x_2$, $x_3\leftrightarrow x_4$ would yield a more extreme value of $\beta$ while $\alpha$ would not change. Without loss of generality, $x_1>0$. Then $x_2,x_3,x_4$ are all non-negative for the portion of $\partial F(A)$ in the upper half plane, and $x_2,x_3<0, x_4>0$ for its portion in the lower half plane. Consequently, the image of a small neighborhood of $x\in f_A^{-1}(0)$ with non-zero $x_2$ will have to miss either the upper or the lower portion of $\partial F(A)$. This proves that $f_A^{-1}$ is not strongly continuous at the origin.

\end{example}

Although the strong continuity of $f_A^{-1}$ fails at $z =0$ in Example \ref{ex:UI4}, it can be shown that weak continuity still holds. The following 6-by-6 example demonstrates that weak continuity of $f_A^{-1}$ can fail in general for unitarily irreducible matrices.

\begin{example} \label{ex:size6}
Let $A = H + iK$ where $H = \text{diag}([0,0,0,1,1,1])$ and $K = \ds \left[ \begin{array}{c|c}  0 & K_1 \\ \hline K_1 & R \end{array} \right]$ where $\ds K_1 = \left[ \begin{array}{rrr} 2 & 0 & 0 \\ 0 & 2 & 0 \\ 0 & 0 & 1 \end{array} \right]$ and $\ds R = \left[ \begin{array}{rrr} 1 & 0 & 1 \\ 0 & 0 & 1 \\ 1 & 1 & 0 \end{array} \right]$. To see that $A$ is unitarily irreducible, it suffices to show that $H$ and $K$ have no common invariant subspaces. Since the eigenvalues of $K$ are distinct, we can simply pick an eigenvector $x$ of $K$ and verify that $\{x,Hx,HKHx,HK^2Hx,K^2Hx,K^3Hx\}$ span $\C^6$. This can be easily done numerically. However we prefer to provide a theoretical justification. If $V$ is an invariant subspace of both $H$ and $K$, then so is $V^\perp$, so if a nontrivial invariant subspace $V$ exists, we may assume that $\dim V \le 3$.  The image $H(V)$ must be a subspace of $V$.  If $\dim H(V) = 3$, then $H(V) = V$ and $V = \text{span} \{e_4,e_5,e_6 \}$.  Clearly, $V$ is not an invariant subspace of $K$, however, so we can rule out this possibility.  If $\dim H(V) = 2$, then since $K_1$ is a nonsingular 3-by-3 block, $(I-H)KH(V)$ must also be a 2 dimensional subspace of $V$ orthogonal to $H(V)$, which contradicts our assumption that $V$ has dimension no more than 3.  If $\dim H(V) = 1$, then both $HKH$ and $HK^2H$ must have $H(V)$ as an eigenspace.  Thus one of the eigenvectors of $R$ corresponds to $H(V)$ as does one of the eigenvectors of $K_1^2 + R^2$.  This cannot be the case since the $K_1$ and $R$ have no common eigenvectors.

Let $x = [x_1,x_2,x_3]^T$, $y = [y_1,y_2,y_3]^T \in \C^3$ with $x^*x = y^*y = 1$, and let $v = [\sqrt{1-\alpha} x,  \sqrt{\alpha} y]^T$ where $\alpha \in [0,1]$.  Then
$$f_A(v) = \alpha + 2 i \sqrt{\alpha - \alpha^2} \real(x^*K_1y) + i\alpha y^*Ry.$$
If $z = \alpha + i \beta \in \partial F(A)$, then $\beta$ corresponds to the maximum or minimum possible values of
\begin{equation} \label{eq:beta}
\imag(f_A(v)) = 2 \sqrt{\alpha - \alpha^2} \real(x^*K_1y) + \alpha y^*Ry
\end{equation}
subject to the constraints $y^*y = x^*x = 1$.  For any given $y$, \eqref{eq:beta} is maximized by choosing $x$ to be a normalized multiple of $K_1 y$, that is $x = \frac{K_1y}{||K_1y||}$.  Then \eqref{eq:beta} becomes
$$
\imag(f_A(v)) = 2 \tfrac{\sqrt{\alpha-\alpha^2}}{||K_1y||} y^*K_1^2y  + \alpha y^*Ry = 2 \sqrt{\alpha-\alpha^2}||K_1 y||  + \alpha y^*Ry.
$$
To minimize \eqref{eq:beta} for a fixed $y$, $x$ must be a negative scalar multiple of $K_1y$, in which case \eqref{eq:beta} becomes
$$
\imag(f_A(v)) = -2 \tfrac{\sqrt{\alpha-\alpha^2}}{||K_1y||} y^*K_1^2y  + \alpha y^*Ry  = -2 \sqrt{\alpha-\alpha^2}||K_1 y||  + \alpha y^*Ry.
$$
For both the maximum and minimum, we may assume that $y \in \R^3$.  Suppose we choose a sequence $\alpha_k \in [0,1]$ such that $\alpha_k \rightarrow 0$.  For each $\alpha_k$, choose $\yp$ which maximizes \eqref{eq:beta} when $\alpha = \alpha_k$ and $\ym \in \R^3$ which minimizes \eqref{eq:beta}.  By passing to a subsequence we can assume that $\yp \rightarrow y^+(0) \in \R^3$ and $\ym \rightarrow y^-(0) \in \R^3$, respectively.  To each $\yp$ we associate $v^+(\alpha_k) = [\sqrt{1-\alpha_k} \yp, \sqrt{\alpha_k} \yp]^T$ and to each $\ym$ we associate $v^-(\alpha_k) = [-\sqrt{1-\alpha_k} \ym, \sqrt{\alpha_k} \ym]^T$.  Then $f_A(v^+(\alpha_k))$ and $f_A(v^-(\alpha_k))$ are the two points on the boundary of $F(A)$ with real part equal to $\alpha_k$.  Note that $||K_1 y||^2 = 2-y_3^2$.  Therefore as $\alpha_k \rightarrow 0$ the ratio $\alpha_k/\sqrt{\alpha_k - \alpha_k^2}$ approaches zero and it follows that $\yp_3 \rightarrow 0$ and $\ym_3 \rightarrow 0$.

Suppose that $y_3 = c$ is fixed.  Then the value of $||K_1y||$ is constant, and to optimize \eqref{eq:beta} we need only find the extreme points of $y^*Ry = y_1^2 + 2c(y_1+y_2)$ subject to the constraint $y_1^2 + y_2^2 = 1-c^2$. It is a standard exercise to show as $c \rightarrow 0$, the maximum converges to $y_1 = 1$, and the minimum always occurs when $y_2 = 1-c^2$.  Therefore $v^+(\alpha_k) \rightarrow e_1$ while $v^-(\alpha_k) \rightarrow e_2$ proving that weak continuity does not hold for $f^{-1}_A$ at $z = 0$.
\begin{figure}[ht]
\includegraphics[scale=1]{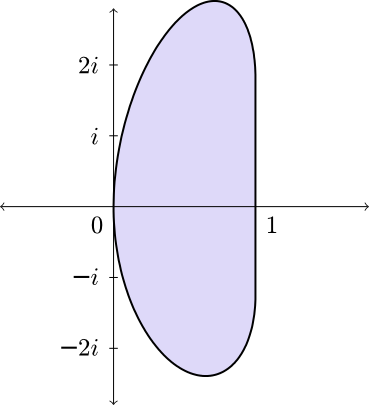}
 
\caption{$F(A)$ from Example \ref{ex:size6}.  Weak continuity fails at $z=0$. }
\end{figure}

\end{example}

\bibliography{master}

\begin{thebibliography}{10}

\bibitem{BraHe94}
V.~Brattka and P.~Hertling.
\newblock Continuity and computability of relations.
\newblock {\em Informatik Berichte}, 164, Fern {U}niversit\"at in {H}agen,
  1994.

\bibitem{Carden09}
R.~Carden.
\newblock A simple algorithm for the inverse field of values problem.
\newblock {\em Inverse Problems}, 25(11):115019, 2009.

\bibitem{Dav71}
C.~Davis.
\newblock The {T}oeplitz-{H}ausdorff theorem explained.
\newblock {\em Canad. Math. Bull.}, 14:245--246, 1971.

\bibitem{GusRa}
K.~E. Gustafson and D.~K.~M. Rao.
\newblock {\em Numerical Range. {T}he Field of Values of Linear Operators and
  Matrices}.
\newblock Springer, New York, 1997.

\bibitem{Hal82}
P.~R. Halmos.
\newblock {\em A {H}ilbert space problem book}.
\newblock Springer-Verlag, New York, second edition, 1982.
\newblock Encyclopedia of Mathematics and its Applications, 17.

\bibitem{Hart04}
K.~P. Hart, J.~Nagata, and J.~E. Vaughan, editors.
\newblock {\em Encyclopedia of general topology}.
\newblock Elsevier Science Publishers B.V., Amsterdam, 2004.

\bibitem{HJ2}
R.~A. Horn and C.~R. Johnson.
\newblock {\em Topics in Matrix Analysis}.
\newblock Cambridge University Press, Cambridge, 1991.

\bibitem{LSS}
C.-K. Li, I.~Spitkovsky, and S.~Shukla.
\newblock Equality of higher numerical ranges of matrices and a conjecture of
  {K}ippenhahn on {H}ermitian pencils.
\newblock {\em Linear Algebra Appl.}, 270:323--349, 1998.

\bibitem{RS11}
L.~Rodman and I.~M. Spitkovsky.
\newblock Ratio numerical ranges of operators.
\newblock {\em Integral Equations and Operator Theory}, 71:245--257, 2011.

\bibitem{Tsi84}
N.-K. Tsing.
\newblock The constrained bilinear form and the {$C$}-numerical range.
\newblock {\em Linear Algebra Appl.}, 56:195--206, 1984.

\end{thebibliography}
\bibliographystyle{plain}

\end{document}